\documentclass[11pt]{article}
\usepackage{amssymb,amsfonts,amsmath,latexsym,epsf,tikz,url}
\usepackage[usenames,dvipsnames]{pstricks}
\usepackage{pstricks-add}
\usepackage{epsfig}
\usepackage{pst-grad} 
\usepackage{pst-plot} 
\usepackage[space]{grffile} 
\usepackage{etoolbox} 
\usepackage{float}
\usepackage{pgfplots}
\usepackage{mathrsfs}
\usetikzlibrary{arrows}
\makeatletter 
\patchcmd\Gread@eps{\@inputcheck#1 }{\@inputcheck"#1"\relax}{}{}
\makeatother

\newtheorem{theorem}{Theorem}[section]

\newtheorem{proposition}[theorem]{Proposition}

\newtheorem{corollary}[theorem]{Corollary}
\newtheorem{lemma}[theorem]{Lemma}
\newtheorem{remark}[theorem]{Remark}

\newtheorem{definition}[theorem]{Definition}

\newcommand{\qed}{\hfill $\square$\medskip}

\begin{document}

\def\nt{\noindent}

\title{Total coalitions in graphs}

\author{	
Saeid Alikhani$^{1}$ 
\and
Davood Bakhshesh$^2$
\and
Hamidreza Golmohammadi$^{3}$  
}


\maketitle

\begin{center}

$^1$Department of Mathematical Sciences, Yazd University, 89195-741, Yazd, Iran\\

$^{2}$Department of Computer Science, University of Bojnord, Bojnord, Iran\\ 

$^3$Novosibirsk State University, Pirogova str. 2, Novosibirsk, 630090, Russia\\
\medskip
{\tt alikhani@yazd.ac.ir ~~ d.bakhshesh@ub.ac.ir ~~h.golmohammadi@g.nsu.ru}
\end{center}

\begin{abstract}
 Let $G$ be a graph with vertex set $V$.  Two disjoint sets $V_1, V_2\subseteq V$ are called a total coalition in $G$, if neither  $V_1$ and $V_2$ is   a total dominating set of $G$ but $V_1\cup V_2$ is a total dominating set.  A total coalition partition of $G$ is a vertex partition $\pi=\{V_1,V_2,\ldots, V_k\}$ such that no set  of $\pi$ is  a total dominating set but  each set $V_i\in \pi$ forms a total coalition with another set $V_j\in \pi$.  The maximum cardinality of a total coalition partition of $G$ is called the total coalition number of $G$, denoted by $TC(G)$.   
In this paper,  we initiate the study of the  total coalition in graphs and its properties.
\end{abstract}

\noindent{\bf Keywords:}  Total dominating set;  Coalition; Total coalition, Tree.

\medskip
\noindent{\bf AMS Subj.\ Class.:}  05C60.


\section{Introduction} 
Let $G = (V,E)$ be  a simple graph with vertex set $V$ and edge set $E$.  The {\em open neighborhood} of a vertex
$v \in V$ is the set $N(v)$ = $\{u | \{u,v\} \in E\}$, and its {\em closed neighborhood} is the set
$N[v]$ = $N(v) \cup \{v\}$. Each vertex  of $ N(v)$ is called a {\em neighbor} of~$v$, and the cardinality of $|N(v)|$
is called the {\em degree} of $v$, denoted by $\deg(v)$.  A vertex~$v$ of degree 1  is called a  {\em leaf}, and its neighbor is called a {\em support vertex}. In a graph $G$ of order $n$, a {\em full vertex} is a vertex of degree $n-1$, and an {\em isolated vertex} is a  vertex of degree $0$. We call a graph  that contains no isolated vertex  an {\em isolate-free} graph.  
The minimum degree and the maximum degree of $G$ is denoted by $\delta(G)$ and  $\Delta(G)$, respectively.   A  set $S\subseteq V$ is called a {\em singleton} if  $|S| = 1$, and is called a {\em non-singleton},  if $|V_i|\ge 2$.
A vertex $v$  of a graph $G$ is called a {\em simplicial} vertex if the graph induced by $N[v]$ is a complete graph.

A set $S\subseteq V$ is called a {\em dominating set} of  $G$, if every vertex of  $V \setminus S$
has at least one  neighbor  in $S$. 
A set $S \subseteq V$ is called a  {\em total dominating set} of $G$ if every vertex of~$V$  has at least one neighbor in $S$. The minimum cardinality of a  total dominating set in~$G$ is called the {\em total domination number} of $G$, denoted by $\gamma_{t}(G)$. A total dominating set with the  cardinality $\gamma_t(G)$ is called a $\gamma_{t}$-set. 
Cockayne, Dawes, and Hedetniemi  introduced the total domination in graphs \cite{3}. To study an overview of research on domination, we refer the reader to  the   books  \cite{7,8,9}. A {\em domatic partition} (or {\em total domatic partition}) is a partition
of the vertex set into dominating sets (or total dominating sets). The maximum cardinality of a domatic partition (or total domatic partition) is called the {\em domatic number} (or {\em total domatic number}), denoted by 
$d(G)$ (or $d_{t}(G)$). The domatic number of a graph was introduced by  Cockayne and
 Hedetniemi \cite{2} and the total domatic number was introduced by Cockayne,
Dawes and Hedetniemi in~\cite{3}. For more details on the domatic number and total domatic number refer to e.g., \cite{11,12,13,14}.

In 2020,  Haynes et al. introduced the  coalitions in graphs  \cite{4}.  For a graph $G$ with vertex set $V$, two sets $V_1,V_2\subseteq V$ form a {\em coalition}, if neither $V_1$ nor $V_2$ is a  dominating set but $V_1\cup V_2$ is a  dominating set in $G$. A {\em coalition partition} in $G$ is a vertex partition $\Psi=\{V_1,\ldots, V_k\}$ such that each set $V_i\in\Psi$  is either a dominating set with $|V_i|=1$,  or there exists another set $V_j\in\Psi$ that forms a coalition with $V_i$.  The maximum cardinality of a coalition partition in $G$ is called the {\em coalition number} of $G$, denoted by $C(G)$. 
Corresponding to any coalition partition $\Psi=\{V_1,V_2,\ldots, V_k\}$ in  a graph $G$, a  {\em coalition graph} $CG(G, \Psi)$ is  associated in which there is a one-to-one correspondence between the   vertices of $CG(G,\Psi)$  and the sets $V_1, V_2,...,V_k$ of $\Psi$, 
and two vertices of  $CG(G,\Psi)$  are adjacent if and only if their corresponding
sets in $\Psi$ form a coalition. 

In  \cite{4}, Haynes et al. provided some upper bounds and lower bounds on the coalition number of a graph. Moreover,  they determined the coalition number of paths and cycles.  In \cite{6}, Haynes et al., provided some other upper bounds on the coalition number of graphs in terms of minimum degree or maximum degree. In \cite{401}, Haynes et al. initiated  the study of the coalition graphs. They proved that any graph is a coalition graph. In   \cite{5}, they characterized the coalition graphs of paths, cycles,  and trees.    
Recently,  Bakhshesh, Henning and Pradhan in \cite{Davood} characterized all graphs $G$ of order $n$ with $\delta(G)\leq 1$ and $C(G)=n$. Moreover, they   characterized all trees $T$ of order $n$ with $C(T)=n-1$.  

As stated in Section 4 of \cite{4} which contains open problems and areas for future research, one can consider total dominating $c$-partition. Motivated by this suggestion,  in this paper,   we introduce the total coalition and total coalition partition.  In the next section, after introducing  the total coalitions,   we establish bounds on the total coalition number.  Moreover, we study the  total coalition number of  graphs $G$ with $\delta(G)=1$ and $\delta(G)=2$ in Sections 3 and 4, respectively. Finally, we conclude the paper in Section 5. 
\section{Introduction to total coalition}
We first state the definition of the total coalitions and then we establish bounds on the total coalition number.
 \begin{definition}[Total coalition]
For a graph $G$ with vertex set $V$,  two sets $V_1,V_2\subseteq V$ form a {total coalition}, if neither $V_1$ nor $V_2$ is a  total dominating set but $V_1\cup V_2$ is a  total dominating set in $G$. 
\end{definition}
Now, we define a total coalition partition in a graph $G$. 
\begin{definition}[ Total coalition partition]
A {total coalition partition} ($tc$-partition) in $G$ is a vertex partition $\pi=\{V_1,\ldots, V_k\}$ such that no set  of $\pi$ is  a total dominating set but  each set $V_i\in \pi$ forms a total coalition with another set $V_j\in \pi$.  The maximum cardinality of a total coalition partition of $G$ is called the {total coalition number} of $G$, denoted by $TC(G)$.  A $tc$-partition of a graph $G$  with the cardinality $TC(G)$ is denoted by $TC(G)$-partition. 
\end{definition}

In the following, we prove that every isolate-free graph $G$ has a  $tc$-partition.
\begin{theorem}\label{1}
 Every isolate-free graph $G$ has a  $tc$-partition.
 \end{theorem} 
\begin{proof}
Let ${\cal D}=\{V_1, V_2,\ldots, V_k \}$ be a total domatic partition of $G$ with $k=d_t(G)$. It is clear that $|V_i|\geq 2$. In the following, we show how to construct a $tc$-partition $\pi$ of $G$.  For any integer  $1\leq i\leq k-1$, if $V_i$ is not a minimal total dominating set, then there is a set $V'_i\subset V_i$ which is a minimal total dominating set. So, we replace $V_i$ by $V'_i$ in $\cal D$, and  add $V_i\setminus V'_i$ to $V_k$.    Hence, we may assume that all sets $V_i\in {\cal D}$ with $1\leq i\leq k-1$ are minimal total dominating sets of $G$. Now, we construct a $tc$-partition $\pi$ for $G$. We start with $\pi=\emptyset$. Now, we  partition every set $V_i\in {\cal D}$ with  $1\leq i\leq k-1$ into  two nonempty sets $V^1_i$ and $V^2_i$, and  then, we add  $V^1_i$ and $V^2_i$ to $\pi$. It is clear that neither  $V^1_i$ nor $V^2_i$ is a total dominating set but $V^1_i\cup V^2_i$  is a total dominating set. Then, clearly $|\pi|\geq 2k-2$. Now, consider the set $V_k\in {\cal D}$. If $V_k$ is a minimal total dominating set, then  we partition $V_k$ into two nonempty sets $V^1_k$ and $V^2_k$, and  then, we add  $V^1_k$ and $V^2_k$ to $\pi$. Clearly,  neither  $V^1_k$ nor $V^2_k$ is a total dominating set but $V^1_k\cup V^2_k$ is a total dominating set. So, this creates a $tc$-partition $\pi$ in $G$ with $|\pi|=2k$. Now, suppose that $V_k$ is not a minimal total dominating set. Then, there is a set $U_k\subset V_k$ which is a minimal total dominating set. Now, we partition $U_k$ into two sets $U^1_k$ and $U^2_k$, and add them to $\pi$.  Clearly, neither  $U^1_k$ nor $U^2_k$ is a total dominating set but $U^1_k\cup U^2_k$ is a total dominating set. Now, consider the set $W_k:=V_k\setminus U_k$.  Since $\cal D$ is a domatic partition of $G$ with the maximum cardinality $d_t(G)=k$, the set $W_k$ is not a total dominating set. Now, if $W_k$ forms a total coalition with  one of the sets of $\pi$, then we add $W_k$ to $\pi$, and therefore $|\pi|=2k+1$.  If $W_k$ does not form a total coalition with any set of $\pi$, then we remove $U^2_k$ from $\pi$ and add $U^2_k\cup W_k$ to $\pi$, and therefore $|\pi|=2k+1$. So, this completes the construction of $\pi$. \qed
\end{proof} 

From  the proof of Theorem \ref{1}, we have  the following results.
\begin{corollary}
\label{cordt}
For any isolate-free graph $G$, it holds that $TC(G)\geq 2d_t(G)$.
\end{corollary}
Assume that a graph $G$ of order $n$ with vertex set $V=\{v_1, v_2,\ldots, v_n\}$ has a full vertex~$v_k$. Since $v_k$ is adjacent to all vertices, it is obvious that $\pi=\left\{\{v_1\}, \{v_2\}, \ldots, \{v_n\}\right\}$ is a $tc$-partition. Then, $TC(G)=n$. So, the following result holds.
\begin{proposition} 
\label{proptcgn}
For any graph $G $  of order $n$ with at least one full vertex, it holds that $TC(G)=n$.
\end{proposition} 
From the proof of Theorem \ref{1} and by Proposition \ref{proptcgn}, it is not hard to conclude the following result. 
\begin{corollary}\label{cor1}
If $G$ is an isolate-free graph of order $n$, then $2\leq TC(G) \leq n$.
\end{corollary}
It is obvious that the total coalition number of the path  $P_2$ is equal to two, and the total  coalition number of the complete graph  $K_n$ is equal to $n$. So, the bounds presented in   Corollary~\ref{cor1} are sharp.

Now, we  need the following theorem.  
\begin{theorem} {\rm \cite{14}}\label{ldet}
If $G$ is a graph of  order $n$ and minimum
degree $\delta(G)$, then 
$d_{t}(G)\ge\lfloor\frac{n}{n-\delta(G)+1}\rfloor$.
\end{theorem}

By Corollary \ref{cordt} and  Theorem \ref{ldet}, we have the following corollary.

\begin{corollary}
If $G$ is  an isolate-free graph of  order $n$, then 
$TC(G)\ge2\lfloor\frac{n}{n-\delta(G)+1}\rfloor$ and this bound is sharp for complete graph $K_n$. 
  \end{corollary}
Now, we prove the following result.
\begin{theorem} \label{2.9}
	If $G$ is an isolate-free  graph with no full vertex, then 
	$ TC(G)\ge \delta(G)+1$ and the bound is sharp. 
\end{theorem} 
\begin{proof} 
Let  $v$ be a vertex of $G$ with the minimum degree $\delta(G)$. Let $k=\delta(G)$, and let $N(v) = \{v_1, v_2,...,v_k\}$. Since $G$ has no full vertex, $V(G)\setminus N[v] \neq\emptyset$. For any integer $1\leq i\leq k$, let $V_i=\{v_i\}$, and let $U=\left(V(G)\setminus N[v]\right)\cup \{v\}$. Let $\pi=\{V_1,V_2,\ldots, V_k, U\}$. It is clear that $\pi$ is a vertex partition of $G$. We claim  that~$\pi$ is a $tc$-partition of $G$. Clearly, all sets $V_i$ are singleton, and therefore, they are not total dominating sets. Since the vertex $v$ has no neighbor in $U$, the set $U$ is not total dominating set. Now, we show that any set of $\pi$ is in total coalition with another set of $\pi$. Consider the set~$U$ and an arbitrary set $V_i$. We show that $S:=V_i\cup U$ is a total dominating set, and therefore, $V_i$ forms a total coalition with $U$. Let $w\in V(G)$. If $w\not\in S$, clearly $w\in N(v)$. Then, $w$ has a neighbor in $S$. Now, suppose that $w\in S$. If $w=v_i$ or $w=v$, since $v_i\in N(v)$, then $w$ has a neighbor in $S$. Now, suppose that $w\in V(G)\setminus N[v]$. If $w$ has a neighbor in $V(G)\setminus N[v]$, then we are done. If $w$ has no neighbor in $V(G)\setminus N[v]$, since $\deg(w)\geq \deg(v)=\delta(G)$, $w$ should be  adjacent to all vertices in $N(v)$. Therefore,~$w$ is adjacent to $v_i$. Hence, $w$ has a neighbor in $S$. This proves the claim. Since $|\pi|=k+1=\delta(G)+1$, it holds that $TC(G)\geq \delta(G)+1$. For the sharpness of the inequality, consider the cycle $C_5$ (see Theorem~\ref{cycle}). \qed
   \end{proof}


Now, we finalize this section with  the following result. 

\begin{theorem}\label{tatmost}
 Let $G$ be a graph with maximum degree $\Delta(G)$,
and let $\pi$ be a $TC(G)$-partition. If $X \in \pi$, then $X$ is in at most $\Delta(G)$  total coalitions.
	    
\end{theorem}
 \begin{proof}
Since $X\in\pi$,  $X$ is not a total dominating set. So, there is a vertex $w$ in $G$ with no neighbor in $X$. Therefore, $X\cap N(w)=\emptyset$. Now, in the following, we consider two cases: $w\in X$ and $w\not\in X$.

 Suppose first that $w\in X$. Now, if a set $A\in \pi$ forms a total coalition with $X$, then  $A\cup X$ is a total dominating set of $G$ and  since $X\cap N(w)=\emptyset$, we should have $A\cap N(w)\neq\emptyset$. Therefore, there are at most $|N(w)|$ sets such as $A$ forming a total coalition with $X$, and consequently, $X$ is in at most $\Delta(G)$ total coalition. 

Now, suppose that $w\not\in X$. So, each set of $\pi$ which is  in total coalition with $X$ must contain at least one of the members of $N[w]$. We claim that there is no set $U\in \pi$ that forms a total coalition with  $X$ and  $U\cap N[w]=\{w\}$. Suppose on
the contrary that there is a set $U\in \pi$ that forms a total coalition with $X$ and  $U\cap N[w]=\{w\}$. So $X\cup U$ is a total dominating set. Then, $w$ has a neighbor in $X\cup U$, which is a contradiction, because $X\cap N(w)=\emptyset$ and $U\cap N(w)=\emptyset$. This proves the claim. Therefore, among all sets of $\pi$  that form a total coalition with $X$  there is exactly one set $U$  with $w\in U$ and $U\cap N(w)\neq \emptyset$.  Hence, there are at most $|N(w)|$ sets of $\pi$ that form a total coalition with $X$. Therefore, $X$ is in total coalition with at most $\Delta(G)$ sets of $\pi$.\qed
 \end{proof}

 \section{Upper bounds and exact values with $\delta(G)=1$}  

 In this section, we present some  upper bounds and exact values for   the total coalition number  of graphs $G$ with $\delta(G)=1$. We first prove the following lemma. 
\begin{lemma}
\label{lem:delta1}
Let $G$ be a graph with $\delta(G)=1$, and let  $x$ be  a leaf of $G$ and $y$ be the support vertex of $x$. Let  $\pi$ be a $TC(G)$-partition, and let $X, Y\in \pi$ such that $x\in X$ and $y\in Y$ (possibly $X=Y$). For any  two sets $A,B\in \pi$ that form a total coalition, we have $A\in \{X, Y\}$ or $B\in\{X,Y\}$. 
\end{lemma}
\begin{proof}
Since $A$ and $B$ form a total coalition,  $A\cup B$ is a total dominating set of $G$. If $A\not\in \{X,Y\}$ and $B\not\in \{X,Y\}$, then the vertex $x$ has no neighbor in $A\cup B$, which is a contradiction. Therefore,  $A\in \{X, Y\}$ or $B\in\{X,Y\}$. \qed
\end{proof}

Now, we prove the following theorem.
\begin{theorem}
\label{thm:lessthann}
If $G$ is a graph of order $n$ with $\delta(G)=1$ and with no full vertex, then $TC(G)<n$.
\end{theorem}
\begin{proof}
Let $x$ and $y$ be two vertices of $G$ with $deg(x)=1$ and $x\in N(y)$. Let $\pi$ be a $TC(G)$-partition.  Suppose on
the contrary that  $TC(G)=n$. So, $\{x\}\in \pi$ and $\{y\}\in \pi$. Suppose that  $\{x\}$ and $\{y\}$  do not form a total coalition.    Since $y$ is the only neighbor of~$x$,   there is no set $\{a\}\in \pi$ with $a\neq y$ forming a total coalition with $\{x\}$. Therefore, $\{x\}$ is not in total coalition with any set of $\pi$, which is a contradiction. So, $\{x\}$ and $\{y\}$   form a total coalition and so,  $\{x,y\}$ is a total dominating set of $G$. Therefore, $y$ is a full vertex, which is a contradiction. Hence,  $TC(G)<n$. \qed
\end{proof}

By Theorem \ref{thm:lessthann}, we have the following result. 
\begin{corollary}
\label{cor:lessthann}
For any tree $T$ of order $n$ with no full vertex, $TC(T)<n$.
\end{corollary}
Note that if a graph $G$ contains a full vertex, by Proposition \ref{proptcgn}, $TC(G)=n$. 

In the following, we show that for any tree $T$ of order $n$ with no full vertex, it holds that $TC(T)<n-1$.
\begin{theorem}
\label{thmn_1}
If $T$ is a tree of order $n$ with no full vertex, then $TC(T)<n-1$.
\end{theorem}
\begin{proof}
By Corollary \ref{cor:lessthann}, $TC(T)\leq n-1$. It suffices  to prove that $TC(T)\neq n-1$. Suppose on the contrary that $TC(T)=n-1$. Let $\pi$ be a $TC(T)$-partition. So, the partition $\pi$ contains a set $A$ of cardinality two, and the rest of the members  of $\pi$ are single-element. Let $x$ be a leaf of $T$ and $y$ be the support vertex of $x$.
Now, we prove  the following claims.
\begin{itemize}
\item{\bf Claim 1.} $A\cap\{x,y\}\neq \emptyset$. Suppose on
the contrary that $A\cap\{x,y\}=\emptyset$. So, $\{x\}\in\pi$ and $\{y\}\in\pi$.  Since $y$ is the only neighbor of $x$, $\{x\}$  forms a total coalition  only with  $\{y\}$. Therefore, by  Lemma \ref{lem:delta1}, all sets of $\pi$ form a total coalition with $\{y\}$ and so $y$ is a full vertex, which is a contradiction.  Therefore,  $A\cap\{x,y\}\neq \emptyset$.
\item{\bf Claim 2.}  $A\neq \{x,y\}$. Suppose on
the contrary that $A=\{x,y\}$. By Lemma~\ref{lem:delta1}, all sets of $\pi$ form a total coalition with $A$. Since the only neighbor of $x$ is the vertex $y$, so $y$ is adjacent to all vertices of $T$ and therefore, $y$ is a full vertex, which is a contradiction.  Therefore, $A\neq \{x,y\}$.
\item {\bf Claim 3.} $x\not\in A$. Suppose on
the contrary that $x\in A$. By Claims 1 and 2, there is a vertex $z\neq y$ in $T$ such that $A=\{x,z\}$, and so $\{y\}\in \pi$.  Since $y$ is the only neighbor of $x$, the set $A$ forms a total coalition only with $\{y\}$, so  $\{x,y,z\}$ is a total dominating set of $T$. Therefore, $y$ is adjacent to $z$. Since $y$ is the only neighbor of $x$,  by  Lemma \ref{lem:delta1}, all sets of $\pi$ form a total coalition with $\{y\}$. Therefore, $y$ is a full vertex, which is a contradiction. Hence, $x\not\in A$. 
\end{itemize}
By Claims 1, 2,  and 3, we should have $y\in A$, and therefore $\{x\}\in \pi$.  Let $u$ be a vertex of $T$ such that $A=\{y,u\}$.  Since $y$ is the only neighbor of $x$, the set $\{x\}$ forms a total coalition only with $A$. So $\{x,y,u\}$ is a total dominating set of $T$. Therefore, $u\in N(y)$. Since $A$ is not a total dominating set, there is a vertex $w$ with $w\neq y$ and $w\neq u$ which is not adjacent to $y$ and $u$.  Moreover,  we have $\{w\}\in \pi$.  By  Lemma \ref{lem:delta1}, $\{w\}$ forms a total coalition  only with $A$. Therefore,   $\{w,y,u\}$ is a total dominating set, and so $w$ must be adjacent to $y$ or $u$, which is a contradiction. So $TC(T)\neq n-1$ and therefore $TC(T)<n-1$. \qed
\end{proof}

Now, we prove the following theorem.
\begin{theorem}
\label{thm:deltaG1}
For any graph $G$ with $\delta(G)=1$,  $TC(G)\leq \Delta(G)+1$.
\end{theorem}
\begin{proof}
Let $x$ and $y$ be two vertices of $G$ with $\deg(x)=1$ and $x\in N(y)$. Let $\pi$ be a $TC(G)$-partition. Suppose first that $W$ is a member of $\pi$ such that contains both  $x$ and $y$. By Lemma \ref{lem:delta1}, for any two sets $A,B\in \pi$ that form a total coalition, $A=W$ or $B=W$. Therefore, all sets of $\pi\setminus\{W\}$ form a total coalition with $W$. By Theorem~\ref{tatmost}, $W$ is in total coalition with at most $\Delta(G)$ sets of $\pi$. Hence, $TC(G)\leq \Delta(G)+1$.

We now suppose that $X$ and $Y$ are two  sets of $\pi$ with $X\neq Y$ such that $x\in X$ and $y\in Y$. By Lemma \ref{lem:delta1}, for any two sets $A,B\in \pi$ that form a total coalition, $A\in\{X,Y\}$ or $B\in\{X,Y\}$.  Since the only neighbor of $x$ is the vertex $y$,  so the set $X$ must be  in total coalition only with the set $Y$.  Hence, all sets of $\pi\setminus\{Y\}$ form a total coalition with $Y$. By Theorem~\ref{tatmost}, $Y$ is in total coalition with at most $\Delta(G)$ sets of~$\pi$. Hence, $TC(G)\leq \Delta(G)+1$.\qed
\end{proof}

By Theorem \ref{thm:deltaG1}, we have the following result. 
\begin{corollary}
\label{col:TDelta}
For any tree $T$, we have $TC(T)\leq \Delta(T)+1$. 
\end{corollary}
In the following, we determine the total coalition number of paths.

\begin{lemma} 
\label{path}
For any path $P_n$ of order $n>1$, it holds that
$$TC(P_n)=\left\{\begin{array}{cc}
2& {\rm if~} n=2 {\rm ~or~} n=4\\
3& {\rm otherwise.}
\end{array}\right.$$
\end{lemma}  
\begin{proof}
For the path $P_2$, it is clear that $TC(P_2)=2$. Now,  by Theorem \ref{thmn_1}, $TC(P_4)\leq 2$. It is not hard to see that $TC(P_4)=2$. Now, we suppose that $n\neq 2$ and $n\neq 4$.  Let $V=\{1,2,\ldots,n\}$ be the vertex set of $P_n$ such that  $\{i,i+1\}$  with $1\leq i<n$ are  the edges of $P_n$. By Corollary \ref{col:TDelta},  we have $TC(P_n)\leq 3$.  Let  $\pi=\left\{\{1\}, \{3\}\right\}\cup \left\{V\backslash\{1,3\}\right\}$. It is clear that $\pi$ is a $TC(P_n)$-partition. Since $|\pi|=3$ and $TC(P_n)\leq 3$, it holds that $TC(P_n)= 3$.\qed
\end{proof}

Based on Lemma \ref{path}, $TC(P_5)=3$ and so  the bound $\Delta(G)+1$ in Theorem \ref{thm:deltaG1} is sharp.

Now, we prove the following theorem.
\begin{theorem}
\label{thmTCG2}
For any isolate-free graph $G$ whose each vertex is adjacent to at least one vertex of degree one, it holds that $TC(G)=2$.
\end{theorem}
\begin{proof}
Since $\delta(G)=1$, so by Theorem  \ref{2.9}, we have $TC(G)\geq 2$. It suffices to prove that $TC(G)\leq 2$.  Suppose on
the contrary that  $TC(G)\geq 3$. Let $\pi$ be a $TC(G)$-partition.   Consider two vertices $v_1$ and $l_1$ such that $l_1$ is a leaf and $l_1\in N(v_1)$.   Now, we consider two cases. 
\begin{itemize}
\item There is a set $U\in \pi$ such that  $\{v_1,l_1\}\subseteq U$.  Since $TC(G)\geq 3$, there are two distinct sets $A,B\in \pi$  with   $A\neq U$ and $B\neq U$.  By Lemma \ref{lem:delta1}, all sets of $\pi\backslash \{U\}$ form a total coalition with $U$ and  any pair of sets of $\pi\backslash \{U\}$ do not form a total coalition. So  both sets $A$ and $B$ form a total coalition with $U$, and $A$ and $B$ do not form a total coalition. We claim that for every vertex $v$ of $G$   with $deg(v)>1$, $v\in U$. Suppose on
the contrary that there is a vertex $v'$ with $\deg(v')>1$ such that $v'\not \in U$. We may assume that $v'\in A$. Now,  let $l'$ be a leaf adjacent to~$v'$. If $l'\in A$, by Lemma \ref{lem:delta1}, any set of $\pi\setminus A$ should  form a total coalition with $A$. Then, $B$ and $A$  form a total coalition which is a contradiction. So,  $l'\not\in A$. Let $X\in \pi$ with $X\neq A$ such that $l'\in X$. Since the only neighbor of $l'$ is $v'$, $X$ must form  a total coalition only with $A$. Now, if $X= U$, then   $X$ forms a total coalition with all sets of $\pi\backslash\{U\}$, which is a contradiction.  But if $X\neq U$,  since any pair of sets of $\pi\backslash \{U\}$ do not form a total coalition,  $X$ does not form a total coalition with $A$, which is a contradiction.  Therefore, $v'\in U$. This proves the claim. Since $U$ contains all vertices $v$ with $deg(v)>1$, $U$ is a total dominating set of $G$, which is a contradiction. Hence, $TC(G)\leq 2$, and  since $TC(G)\geq 2$, we have  $TC(G)=2$.
\item  There is no set  $U\in \pi$  such that  $\{v_1,l_1\}\subseteq U$. Let $C,D\in \pi$ with $C\neq D$ such that $v_1\in C$ and $l_1\in D$.  Since the only neighbor of $l_1$ is the vertex $v_1$, by Lemma~\ref{lem:delta1}, all sets of $\pi\backslash \{C\}$ form a total coalition only with $C$,  and  any pair of sets of $\pi\backslash \{C\}$ do not form a total coalition.  We claim that for every vertex $v$ with $deg(v)>1$, $v\in C$. From here, by the identical arguments to  the previous case,  the claim is proved.  It is enough to assume that $C$ is the same as $U$.  Therefore,  the claim is proved. Since $C$ contains all vertices $v$ with $deg(v)>1$,  $C$ is a total dominating set of $G$, which is a contradiction. Hence, $TC(G)\leq 2$, and  since $TC(G)\geq 2$, we have  $TC(G)=2$.
\end{itemize}
This completes the proof. \qed
\end{proof}

The corona product of two graphs $F$ and $H$,  denoted by $F\circ H$,  is defined as the graph obtained by taking one copy of $F$ and $|V(F)|$ copies of $H$ and joining the $i$-th vertex of $F$ to every vertex in the $i$-th copy of $H$.  By Theorem \ref{thmTCG2}, we easily conclude the following result.
\begin{corollary}
For any graph $G$,  $TC(G\circ\overline{K_n})=2$.
\end{corollary}

\section{Upper bounds and exact values with $\delta(G)=2$}
Here, we prove that for any graph $G$ with $\delta(G)=2$,  $TC(G)\leq 2\Delta(G)$. Then, we determine the exact  value of the total coalition number of cycles.  We start with the following lemma whose proof is similar  to the proof of Lemma \ref{lem:delta1}.  
\begin{lemma}
\label{lem:delta2}
Let $G$ be a graph with $\delta(G)=2$, and let  $x$ be  a vertex of $G$ with $deg(x)=2 $, and suppose that $N(x)=\{y,z\}$, where $y$ and $z$ are vertices of $G$. Let  $\pi$ be a $TC(G)$-partition, and let $X, Y, Z\in \pi$ such that $x\in X, y\in Y$ and $z\in Z$ (possibly $X=Y$ or $Y=Z$ or $X=Z$). For any  two sets $A,B\in \pi$ that form a total coalition, we have $A\in \{X, Y, Z\}$ or $B\in\{X,Y, Z\}$. 
\end{lemma}
Now, we prove the following theorem.
\begin{theorem}
\label{delta2}
For any graph  $G$ with $\delta(G)=2$, $TC(G)\leq 2\Delta(G)$.
\end{theorem}
\begin{proof}
Let $x$ be a vertex of $G$ with $\deg(x)=2$, and suppose that $N(x)=\{y,z\}$, where $y$ and $z$ are vertices of $G$. Let $\pi$ be a $TC(G)$-partition. Now, we consider the following cases.
\begin{itemize}
\item{\bf Case 1.} There is a set $U\in \pi$ such that $\{x,y,z\}\subseteq U$. By Lemma \ref{lem:delta2}, all sets of $\pi\backslash U$ form a total coalition with $U$. So, by Theorem \ref{tatmost}, $U$ is in total coalition with at most $\Delta(G)+1$ sets. Therefore, $TC(G)\leq \Delta(G)+1+1\leq 2\Delta(G)$. 
\item{\bf Case 2.} There are two distinct sets $X, A\in \pi$  such that $x\in X$ and $\{y,z\}\subseteq A$.  Since $N(x)\subseteq A$, there is no set $B\in \pi$ with $B\neq A$  that  forms a total coalition with~$X$. So $X$ forms a total coalition  only with $A$. Moreover, by Lemma \ref{lem:delta2}, all sets of $\pi\backslash \{A\}$ form a total coalition with  $A$.  Therefore, by Theorem \ref{tatmost}, $A$ is in total coalition with at most $\Delta(G)$ sets. So, $TC(G)\leq \Delta(G) +1\leq 2\Delta(G) $.
\item{\bf Case 3.} There are two  distinct sets $Y, B\in \pi$  such that $y\in Y$ and $\{x,z\}\subseteq B$. By Lemma~\ref{lem:delta2}, all sets of $\pi\backslash\{Y,B\}$  form a total coalition with $Y$ or $B$. Now, suppose that $Y$ and $B$ form a total coalition.  By Theorem \ref{tatmost}, each set  of $Y$ and $B$ is in at most $\Delta(G)$ total coalition. Therefore, $TC(G)\leq \Delta(G)-1+\Delta(G)-1+1+1=2\Delta(G)$. Now, suppose that $Y$ and $B$ do not form a total coalition. Since $Y\cup B$ is not  a total dominating set,  there is a vertex $w$ in $G$ that has no neighbor in $Y\cup B$. Since  all sets of $\pi\backslash\{Y,B\}$  form a total coalition with $Y$ or $B$, for totally dominating the vertex $w$, all sets  of $\pi\backslash\{Y,B\}$ must contain at least one members of $N(w)$. Hence, $TC(G)\leq |N(w)|+2\leq \Delta(G)+2\leq 2\Delta(G)$.
\item {\bf Case 4.} There are two distinct sets $Z, C\in \pi$  such that $z\in Z$ and $\{x,y\}\subseteq C$.
The proof is similar to the proof of  {\bf Case 3}.
\item {\bf Case 5.} There are three distinct sets $X, Y, Z\in \pi$ such that $x\in X, y\in Y$ and $z\in Z$.
Since the neighbors of  the vertex $x$ are in $Y$ and $Z$,  the set $X$ forms a total coalition only with $Y$ or $Z$. By Lemma~\ref{lem:delta2}, all sets of $\pi\backslash\{Y, Z\}$  form a total coalition with $Y$ or $Z$. Suppose that $Y$ and $Z$ do not form a total coalition. So $Y\cup Z$ is not a total dominating set. Therefore, there is a vertex $w$ with no neighbor in $Y\cup Z$. Since all sets of $\pi\backslash\{Y, Z\}$  form a total coalition with $Y$ or $Z$, for totally dominating the vertex $w$, all sets of $\pi\backslash\{Y, Z\}$  must contain at least one of the members of $N(w)$.  Now, since  $X$ forms a total coalition with at least  one the sets $Y$ and $Z$,   $TC(G)\leq |N(w)|-1+3\leq \Delta(G)+2\leq 2\Delta(G)$. 

Now,  suppose that $Y$ and $Z$ form a total coalition. Suppose first  that $X$ forms a total coalition with exactly one of the sets $Y$ and $Z$. Suppose, without loss of generality, that $X$ forms a total coalition  with  $Z$, and does  not form a total coalition with $Y$.   By Theorem \ref{tatmost}, each set  of $Y$ and $Z$ is in at most $\Delta(G)$ total coalition. So $TC(G)\leq \Delta(G)-1+\Delta(G)-2+1+1+1=2\Delta(G)$. Now, if $X$ forms a total coalition with both  sets $Y$ and $Z$, then by Theorem \ref{tatmost}, we obtain that $TC(G)\leq \Delta(G)-2+\Delta(G)-2+1+1+1=2\Delta(G)-1\leq 2\Delta(G)$.\qed
\end{itemize}
\end{proof}
Now, we prove the following result.
\begin{theorem}\label{cycle}
For any cycle $C_n$, 
$$ TC(C_{n})=\left\{
 \begin{array}{cc}
 4  &\quad n\equiv 0 ~(\mbox{mod } 4)\\
 3   &\quad\mbox{otherwise. }
 \end{array}\right.
 $$
\end{theorem}  
\begin{proof}
By Theorem \ref{delta2}, we have $TC(C_n)\leq 4$. Let $V=\{1,2,\ldots, n\}$ be the vertices of $C_n$, and $\{i,i+1\}$ with $1\leq i<n$ and $\{n,1\}$ are the edges of $C_n$. Assume that $n\equiv 0~(\mbox{mod } 4)$. A $TC(C_n)$-partition $\pi=\{A, B, C, D\}$  is constructed as follows. 

Suppose that $\frac{n}{4}$ is an even number. So,
$$A=\bigcup_{k=0}^{\frac{n-8}{8}}\{4k+1, 4k+2\},~B=\bigcup_{k=\frac{n}{8}}^{\frac{n-4}{4}}\{4k+1, 4k+2\},$$
$$C=\bigcup_{k=0}^{\frac{n-8}{8}}\{4k+3, 4k+4\},~D=\bigcup_{k=\frac{n}{8}}^{\frac{n-4}{4}}\{4k+3, 4k+4\},$$
Now, suppose that $\frac{n}{4}$ is an odd number. So,
$$A=\left(\bigcup_{k=0}^{\left \lfloor\frac{n-8}{8}\right\rfloor}\{4k+1, 4k+2\}\right)\bigcup\left\{4\left \lceil\frac{n-8}{8}\right\rceil+1\right\},$$
$$B=\left(\bigcup_{k=\left\lceil\frac{n}{8}\right\rceil}^{\frac{n-4}{4}}\{4k+1, 4k+2\}\right)\bigcup\left\{4\left \lceil\frac{n-8}{8}\right\rceil+2\right\},$$
$$C=\left(\bigcup_{k=0}^{\left \lfloor\frac{n-8}{8}\right\rfloor}\{4k+3, 4k+4\}\right)\bigcup\left\{4\left \lceil\frac{n-8}{8}\right\rceil+3\right\},$$
$$D=\left(\bigcup_{k=\left\lceil\frac{n}{8}\right\rceil}^{\frac{n-4}{4}}\{4k+3, 4k+4\}\right)\bigcup\left\{4\left \lceil\frac{n-8}{8}\right\rceil+4\right\}.$$
It is not hard to verify that $A$ and $B$ form a total coalition,  and $C$ and $D$ form a total coalition.

Now, suppose that $n$ is not divisible by 4. We show that $TC(C_n)\neq 4$. Suppose on
the contrary that $TC(C_n)=4$. Let $\pi=\{A, B, C, D\}$ be the $TC(C_n)$-partition.  
By Theorem \ref{tatmost}, each set of $\pi$ is in total coalition with  at most two sets of $\pi$.
So, we  assume, without loss of generality, that $A$ and $B$ form a total coalition, and $C$ and $D$ form a total coalition. It is not hard to see that for any  integer $n>0$ which is not divisible by 4, it holds that $\gamma_t(C_n)=\lfloor\frac{n}{2}\rfloor+1$.  Hence, since $A\cup B$ and $C\cup D$ are the total dominating sets of $C_n$, it holds that $|A|+|B|\geq \lfloor\frac{n}{2}\rfloor+1 $ and $|C|+|D|\geq \lfloor\frac{n}{2}\rfloor+1 $. Therefore, $|A|+|B|+|C|+|D|\geq 2\lfloor\frac{n}{2}\rfloor+2 $. On the other hand, since $\pi$ is a vertex partition of $C_n$,   $|A|+|B|+|C|+|D|=n$. Then, we have $n\geq 2\lfloor\frac{n}{2}\rfloor+2$, which is a contradiction. Hence, $TC(C_n)\neq 4$. Then, $TC(C_n)\leq 3$. In the following, we present a $TC(C_n)$-partition $\pi$ such that $|\pi|=3$. To do this, let  $\pi$ = $\{A=\{1,2\}, B=\{3,4\}, C=V(C_n)\setminus(A\cup B)\}$. Observe that each of $A$ and $B$ form a total coalition with $C$. \qed
\end{proof}

\begin{remark}
 Let $n$ be any positive integer and $F_n$ be the Friendship graph with $2n+1$ vertices and $3n$ edges, formed by the join of $K_1$
 with $nK_2$. Since $F_n$ has a full vertex,  $TC(F_n)=2n+1$ and  $|2\Delta(F_n)-TC(F_n)|=|2n-1|$. This shows that there is a graph $G$ such that the value
 of $|2\Delta(G)-TC(G)|$ can be 
 arbitrarily large.  
\end{remark}

\section{Conclusion}
 In this paper, we have introduced the total coalition concept in graphs and we have studied some properties for the total coalition number. We proved that for any graph $G$ with $\delta(G)=1$, $TC(G)\leq \Delta(G)+1$,  and for any graph~ $G$ with $\delta(G)=2$, $TC(G)\leq 2\Delta(G).$ Using these bounds, we obtained the exact values of $TC(P_n)$, $TC(G\circ \overline{K_n})$  and $TC(C_n)$. There are many open problems and areas in the study of the total coalition number  of a graph that we  state and close the paper with some of them. 
 \begin{enumerate}
 \item What is the total coalition number of  graph operations, such as corona, Cartesian, join, lexicographic, and so on?

 \item What is the total coalition number of  natural and fractional powers of a graph (see e.g. \cite{BIMS})?

\item What is the effects on $TC(G)$ when $G$ is modified by operations on vertex and edge of $G$?
 \item Similar to the coalition graph of $G$, it is natural to define and study the total coalition graph of $G$ for total coalition partition $\pi$, which can be denoted by $TCG(G,\pi)$, and is defined as follows. Corresponding to any total coalition partition $\pi=\{V_1,V_2,\ldots, V_k\}$ in  a graph $G$, a  {\em total coalition graph} $TCG(G, \pi)$ is  associated in which there is a one-to-one correspondence between the   vertices of $TCG(G,\pi)$  and the sets $V_1, V_2,...,V_k$ of $\pi$, 
and two vertices of  $TCG(G,\pi)$  are adjacent if and only if their corresponding
sets in $\pi$ form a total coalition. 
 
\item Study the complexity of the total coalition number for many of the graphs.  

 \end{enumerate} 
\medskip

\nt{\bf Acknowledgement.} 
The work of Hamidreza Golmohammadi is supported by the Mathematical Center in Akademgorodok, under agreement No. 075-15-2022-282 with the Ministry of Science and High Education of the Russian Federation.

\end{document}